\documentclass[12pt]{amsart}
\usepackage{graphicx}
\usepackage{amsmath}
\usepackage{amsfonts}
\usepackage{amssymb}
\usepackage[mathscr]{eucal}
\makeatletter
\def\numberwithin#1#2{\@ifundefined{c@#1}{\@nocnterrr}{
  \@ifundefined{c@#2}{\@nocnterr}{
  \@addtoreset{#1}{#2}
  \toks@\expandafter\expandafter\expandafter{\csname the#1\endcsname}
  \expandafter\xdef\csname the#1\endcsname
    {\expandafter\noexpand\csname the#2\endcsname
     .\the\toks@}}}}
\makeatother
\numberwithin{equation}{section}

\newtheorem{theorem}{Theorem}
\numberwithin{theorem}{section}

\newtheorem{proposition}[theorem]{Proposition}
\newtheorem{remark}[theorem]{Remark}

\begin{document}

\title{Ancient solutions for Andrews' hypersurface flow}

\author{Peng Lu}
\address{Department of Mathematics, University of Oregon, Eugene, OR 97403}
\email{penglu@uoregon.edu}

\author{Jiuru Zhou}
\address{School of Math. Science \\
Yangzhou University \\
Yangzhou,  Jiangsu 225002, China}
\email{zhoujiuru@yzu.edu.cn}


\date{\small revised \today}

\begin{abstract}
We construct the ancient solutions of the hypersurface flows in Euclidean spaces
studied by B. Andrews in 1994.
As time $t \rightarrow 0^-$ the solutions collapse to a round point where $0$ is the singular time.
But as $t\rightarrow-\infty$ the solutions become more and more oval.
Near the center the appropriately-rescaled pointed Cheeger-Gromov limits are round cylinder solutions
$S^J \times \mathbb{R}^{n-J}$, $1 \leq J \leq n-1$.
These results are the analog of the corresponding results in Ricci flow ($J=n-1$) and mean curvature flow.

\smallskip
\noindent \textbf{Keywords}. Andrews' hypersurface  flow, ancient solutions, asymptotic limits

\smallskip
\noindent \textbf{MSC (2010)}. 53C44, 35K55, 58J35
\end{abstract}

\maketitle

\section{Introduction}

Ancient solutions are studied in various geometric flows including Ricci flow (\cite{Pe02},
\cite{DHS12}), curve shortening flow (\cite{DHS10}), and mean curvature flow (\cite{Wa11},
\cite{Wh03}, \cite{HH16}) because of their close relation with  the singularity analysis.
In this article we use Perelman's idea in the construction of ancient oval solutions for Ricci flow
and some ideas from \cite{HH16}
 to construct ancient oval solutions for the Andrews' flow (\cite{An94}).
Note that for mean curvature flow such a construction is done in  the nice work of Haslhofer and
Hershkovits (\cite{HH16}). 
The uniqueness of the oval solutions for mean curvature flow is proved recently by Angenent, P. Daskalopoulos, and and N. Sesum (\cite{ADS18}).

Now we recall the definition of Andrews' flow.
Let $M^n$ be a smooth manifold of dimension $n \geq 2$ without boundary.
  For a smooth family of immersions $\varphi: M \times  [0,T) \rightarrow \mathbb{R}
^{n+1}$ we denote hypersurface  $\varphi(M \times \{t \})$ in $\mathbb{R}^{n+1}$
by $\Sigma_t^n$ and denote the unit normal vector of $\Sigma_t$ at $\varphi(p,t)$ by
$\nu (\varphi(p,t), t)$ (chosen to be  outward-pointing if the hypersurface is convex).
We denote the principal curvatures $\Sigma_t^n$ at point $\varphi(p,t)$
by $\lambda_{t,1}(\varphi(p,t)) \leq \cdots \leq \lambda_{t,n}(\varphi(p,t))$.
Recall that a hypersurface is strictly convex if its principal curvatures are positive.

Define the cone $\Gamma_+^n =\{ (\lambda_1, \cdots, \lambda_n ) \in \mathbb{R}^n, \,
\forall i, \lambda_i >0 \}$.  Let $f$ be a smooth positive function on $\Gamma_+^n$. Throughout
this article we assume that $f$ is  symmetric, homogeneous of degree one, and $\frac{\partial f}
{\partial \lambda_i} >0$ for each $i$.  We further make the following assumption for $f$ (\cite{An94},
\cite{An07}).

\noindent \textbf{Assumption A}. Let function $f^*$ be defined by
$f^*(\lambda_1, \cdots, \lambda_n )  =- f( \lambda_1^{-1}, \cdots, \lambda_n^{-1} ) $.
Function $f$ is assumed to satisfy one of the following assumption
\begin{enumerate}
\item[(A1)] $n=2$, or

\item[(A2)] $f$ is convex, or

\item[(A3)] $f$ is concave on $\Gamma_+^n$ and $f$  approaches zero on the boundary
of $\Gamma_+^n $, or

\item[(A4)] both $f$ and  $f^*$ are concave $\Gamma_+^n$.
\end{enumerate}

 We consider the following equation of $\varphi$
\begin{equation}
 \frac{\partial}{\partial t} \varphi (p,t) = - f(\lambda_{t,1}(\varphi(p,t)), \cdots, \lambda_{t,n}
 (\varphi(p,t))) \cdot \nu (\varphi(p,t),t).   \label{eq Andrews flow}
\end{equation}
We call the flow (\ref{eq Andrews flow}) the Andrews' flow.
The solutions of Andrews' flow are necessarily strictly convex hypersurfaces.
Below we use $S^n $ to denote the unit sphere and  $S^n(R) \subset \mathbb{R}^{n+1}$
to denote the round sphere of radius $R$ centered at the origin.

\begin{theorem} \label{thm main}
Fixing a $J \in \{ 1,2, \cdots, n-1 \}$, we assume that the Condition $B_J$ holds for function $f$
(see  \S \ref{subsec compar cylinder} for the definition),
then there exists a non-spherical $O_{J+1} \times O_{n-J}$-rotationally symmetric
ancient solution $\varphi_J: S^n \times
(- \infty, 0) \rightarrow \mathbb{R}^{n+1}$  to Andrews' flow (\ref{eq Andrews flow})
 which has strictly convex time slices $\hat{\Sigma}_{J, \infty, t}$ and which develops
 a singularity at time $0$.
As $t \rightarrow 0-$ the solutions collapse to a round point.
\end{theorem}

\begin{proposition} \label{prop asym back limit cylinder}
Assume $f$ in  (\ref{eq Andrews flow}) satisfies Assumption E (see \S \ref{sec 3 condion E} for the definition).
Then the solutions $\varphi_J$ constructed in Theorem \ref{thm main}  have
the following backwards asymptotics:  for $K \rightarrow
\infty$ the parabolically-rescaled  flows $K^{-1} \hat{\Sigma}_{J, \infty, K^2t}$ sub-converge
to the round shrinking cylinder  $S^J( \sqrt{2J |t|}) \times \mathbb{R}^{n-J}$
 for $t \leq -1$.
\end{proposition}

\begin{remark}
Corresponding to the choice of $J=0$ and for mean curvature flow Bourni, Langford and Tinglia construct compact convex and collapsing ancient solutions that lie in a slab with $O(1)\times O(n)$-symmetry (\cite{BLT17}, see also \cite{HIMW} and \cite{Wa11}). They further study the uniqueness of such solutions.
\end{remark}

\vskip .1cm
\noindent \textbf{Acknowledgement}.
The authors would like to thank Mat Langford for bringing our attention to \cite{BL16} and \cite{BLT17}. J.R.Z. would like to thank China Scholarship Council for providing a fellowship as a visiting scholar at University of Oregon. P.L. is partially supported by Simons Foundation through Collaboration Grant 229727 and 581101. J.R.Z. is partially supported by a PRC grant NSFC 11426195.

\section{Comparison principle for Andrews' flow} \label{sec 2 comparison}

It is well-known that mean curvature flow satisfies a powerful comparison priciple (\cite{Ec04}).
We will need a comparison principle for Andrews' flow (\cite[Proof of Theorem 6.2]{An94}
and \cite[Theorem 5]{ALM13}) later.

\subsection{The comparison principle}
Before we state it, we need a little preparation. When we write  $\varphi (p,t)$ in (\ref{eq Andrews flow})
as graph of function $u(\tilde{x},t)$, i.e.,  $\varphi (p,t) = (\tilde{x}(p,t), u( \tilde{x}(p,t), t))$ with
 unit normal direction  $\nu(\tilde{x},t) = \nu(\varphi(p,t),t)=\frac{1}{\sqrt{|Du|^2 +1}}(-Du,1)$ where
 $Du(\tilde{x},t) = (\partial_{\tilde{x}^1} u
(\tilde{x},t), \cdots, \partial_{\tilde{x}^n} u (\tilde{x},t))$, the equation (\ref{eq Andrews flow}) becomes
\begin{align*}
\begin{cases}
\frac{\partial \tilde{x}}{\partial t} =  \frac{f(\lambda_{t,1}, \cdots, \lambda_{t,n})}{
\sqrt{| Du|^2 +1}} \cdot Du, \\
\langle Du , \frac{\partial \tilde{x}}{\partial t} \rangle  + \frac{\partial u}{\partial t} = - \frac{
f(\lambda_{t,1}, \cdots, \lambda_{t,n})}{\sqrt{| Du|^2 +1}},
\end{cases}
\end{align*}
where $\lambda_{t,i} = \lambda_{t,i}(\tilde{x}) = \lambda_{t,i} (\tilde{x}(p,t), u( \tilde{x}(p,t), t))$.
It follows that the graph function $u(\tilde{x},t)$ defined on some open subset of $\mathbb{R}^n
\times [0,T)$ satisfies
\begin{equation}
 \frac{\partial u}{\partial t} = -
 f(\lambda_{t,1}, \cdots, \lambda_{t,n}) \sqrt{| Du|^2 +1}. \label{eq Andrews flow in u func}
\end{equation}

Below we will use the following convention for the second fundamental form of graph hypersurface
 $( \tilde{x}, u(\tilde{x}))$ as used in \cite[(2.5)]{An94},
$h =(h_{AB}) = \left ( - \frac{D_{AB}^2 u}{\sqrt{|Du|^2 +1}} \right )$ where $D_{AB}^2 u (\tilde{x}) =
\partial_{\tilde{x}^A} \partial_{\tilde{x}^B} u(\tilde{x})$.

\begin{proposition}[B. Andrews]  \label{prop comparison prin}
Let $\varphi_a: M^n_a \times [0,T) \rightarrow \mathbb{R}^{n+1}, \, a=1, 2$, be two solutions
of Andrews' flow (\ref{eq Andrews flow}) with strictly convex initial hypersurfaces $\Sigma_{a,0}$ where
  $\Sigma_{a,t} = \varphi_a(M_a\times \{t \}) \subset \mathbb{R}^{n+1}$.
We further assume that $M_1$ is closed, that $M_2$ with induced metric by map $\varphi_{2}
(\cdot, t)$ is complete for each $t$, that map $\varphi_{2} (\cdot, t): M_2 \rightarrow \mathbb{R}
^{n+1}$ is proper for each $t$, and that the convex hull of  $\Sigma_{2,0}$ contains  $\Sigma_{1,0}$.
Then the distance between hypersurfaces $\Sigma_{1,t}$ and $\Sigma_{2,t}$ is non-decreasing
in time $t$.
\end{proposition}

\begin{proof}
 Let $\rho (t)$ be the distance between $\Sigma_{1,t}$  and $\Sigma_{2,t}$.
 Below we will use Hamilton's trick to prove (super-right) derivative
$\frac{d}{dt} \rho(t) \geq 0$ whenever $\rho (t) >0$, then the proposition follows.

Fix a $t \in [0,T)$ with $\rho (t) >0$. By  the assumption the distance $\rho (t)$
is attained by points $x^0_{1,t} \in \Sigma_{1,t}$  and  $x^0_{2,t} \in \Sigma_{2,t}$,
$\rho (t) =|x^0_{1,t} - x^0_{2,t}|$. The tangent planes $T_{x^0_{1,t}} \Sigma_{1,t}$ and
 $T_{x^0_{2,t}} \Sigma_{2,t}$ are parallel. Hence $\Sigma_{a,t}$ locally (close to $x^0_{a,t}$)
 can be written as graphs over  some common small ball in $\mathbb{R}^n$.
 Without loss of generality we may assume that the graph functions are $u_a(\tilde{x}, t)$ over open
 ball $B^n_0(r)$ of center $0$ and radius $r$ such that $x^0_{a,t}=(0, u_a(0,t))$. From the assumption
 we may also assume that $u_1(0,t)  < u_2(0,t)$, that the
normal vectors are $\nu_a( \tilde{x},t)=\frac{1}{\sqrt{|Du_a|^2 +1}}(-Du_a,1)$,
 and that $\Sigma_{a,t} \cap B^{n+1}_{x^0_{a,t}}(r_*) \subset
 \{ (\tilde{x}, u_a( \tilde{x},t)), \,  \tilde{x} \in B^n_0(r) \}$ for some $r_* >0$ small enough.

By the construction above for the particular $t$ we have $\rho (t) = u_2(0,t) - u_1(0,t)$ and function
$u_2(\cdot,t) - u_1(\cdot,t)$ on $B^n_0(r)$ has a  positive local minimum at $\tilde{x} =0$. Hence
\begin{equation*}
 \left . D (u_2(\tilde{x},t) - u_1(\tilde{x},t)) \right |_{\tilde{x}=0} =0, \quad
 \left . D^2 (u_2(\tilde{x},t) - u_1(\tilde{x},t)) \right |_{\tilde{x}=0} \geq 0.
\end{equation*}
Let $g_{a,t}(\tilde{x})$ and $h_{a,t}(\tilde{x})$ be the induced metric and second fundamental form of
hypersurface $\Sigma_{a,t}$ at $(\tilde{x}, u_a( \tilde{x},t))$
using the canonical choice of unit normal vectors, respectively.
Hence we have that
\begin{equation} \label{eq 2fund test}
g_{1,t} (0) =g_{2,t} (0) \quad \text{ and } \quad  h_{2,t} (0) \leq h_{1,t} (0).
\end{equation}

By \cite[Theorem 1.1]{An07} the hypersurfaces $\Sigma_{a,t}$ are strictly convex for each $t$, i.e.,
the principal curvatures $\lambda_{a,t,i}( \tilde{x}) >0$ for each $a$, $t$,  $i$, and $\tilde{x}$.
We compute at the $t$ by using  the Hamilton's trick and equations (\ref{eq Andrews flow in u func})
and (\ref{eq 2fund test})
\begin{align*}
\frac{d}{dt} \rho (t) & = \left . \frac{\partial (u_2-u_1)}{\partial t} \right |_{\tilde{x}=0} \\
&=\left . \left (-  f( \lambda_{2,t,1}, \cdots,  \lambda_{2,t,n}) \sqrt{1+| Du_2|^2}
+ f( \lambda_{1,t,1}, \cdots,  \lambda_{1,t,n}) \sqrt{1+| Du_1|^2}  \right ) \right |_{\tilde{x}=0} \\
& = \left . \left (  f( \lambda_{1,t,1}, \cdots,  \lambda_{1,t,n}) - f( \lambda_{2,t,1},
 \cdots,  \lambda_{2,t,n})  \right ) \sqrt{1+| Du_1|^2} \right |_{\tilde{x}=0}.
\end{align*}

Below we use $S$ to denote linear subspace of tangent space
$T_{x^0_{a,t}} \Sigma_{a,t}$. By Courant minmax principle we have the principal curvature
\[
\lambda_{a,t,i} (0) = \min_{\operatorname{dim}S =i} \,\, \max_{y \in S, \, g_{a,t}(y,y) =1}
 y h_{a,t}(0) y^T,
\]
 it follows from  (\ref{eq 2fund test}) that
$0< \lambda_{2,t,i} (0) \leq \lambda_{1,t,i} (0) $ for each $i$.
Recall that $f$ is assumed to be strictly monotone increasing in each argument,
 we conclude  that $ f(\lambda_{2,t,1}, \cdots,
\lambda_{2,t,n}) \leq f( \lambda_{1,t,1},  \cdots,  \lambda_{1,t,n})$ at $\tilde{x}=0$.
We have proved $\frac{d}{dt} \rho (t) \geq 0$.
\end{proof}

\subsection{Comparing with cylindrical solutions} \label{subsec compar cylinder}

Since in our construction of ancient solutions we need to compare solutions with
cylindrical solutions which are not strictly convex. Here we pay some attention to the
difference with Proposition \ref{prop comparison prin}.

Fixing a $J \in \{ 1, 2, \cdots, n-1 \}$, we define the cone $\Gamma^{n}_{+,J} =\{ (\lambda_1,
\cdots, \lambda_n )
\in \mathbb{R}^n, \, \lambda_i \geq 0 \text{ for } 1 \leq i\leq n-J, \text{ and } \lambda_i>0
\text{ for } i \geq n-J+1\}$.
Now we introduce another condition on function $f$.

\vskip .1cm
\noindent \textbf{Assumption $B_J$}. Given $J \in \{1, 2, \cdots, n-1 \}$ $f$ can be extended to a continuous
function on $\Gamma^n_{+,J}$ and $c_{J0}
= f(\underbrace{0, \cdots, 0}_{n-J},  1, \cdots, 1) >0$.

The principle curvatures of hypersurface $S^{J}(R) \times \mathbb{R}^{n-J}$
are  $\lambda_i=0$ for $i=1,\cdots,n-J$ and  $\lambda_i=\frac{1}{R}, i=n-J+1,\cdots,n$.
Hence if $f$ satisfies Assumption $B_J$, it is easy to check that $S^{J}(R(t)) \times \mathbb{R}^{n-J}$
with $R(t) =\sqrt{(R(0))^2 -2c_{J0}t} $ is a solution of (\ref{eq Andrews flow})
where the corresponding map is
\begin{equation} \label{eq cylder sol j}
\varphi_J (y,z,t) =(R(t)y,z) \quad \text{ for } (y,z) \in S^{J}(1)
\times \mathbb{R}^{n-J} \text{ and } t< T_J,
\end{equation}
 where $T_J=\frac{(R(0))^2}{2c_{J0}}$ is the singular time of the solutions.
 Note that the unit normal vector is $\nu(R(t)y,z,t) =(y,0)$.

\begin{remark} \label{rk comar cylind}
Fix a $J \in \{1, \cdots, n-1 \}$, if we further assume that $f$ in Proposition \ref{prop comparison prin}
satisfies Assumption $B_J$,
then the conclusion in the proposition still holds when we take hypersurface $\Sigma_{2,t}$
 to be $S^{J}(R(t))\times\mathbb{R}^{n-J}$. The proof is trivial.
\end{remark}

\section{Compactness theorem for Andrews' flow}\label{sec 3 condion E}

The proof of the following compactness theorem is similar to the proof of
\cite[Theorem 6.1]{An94} (compare to \cite[Theorem 17]{AMZ13}).

\begin{theorem} \label{thm compactness Andrews flow}
Let $\varphi_a: S^n \times (\alpha_a, 0] \rightarrow \mathbb{R}^{n+1}$ be a sequence of
solutions to Andrews' flow (\ref{eq Andrews flow}) with $\lim_{a \rightarrow \infty}
\alpha_a = \alpha_{\infty} <0$.
Assume that for any $t_* \in ( \alpha_{\infty},0]$ there is a constant $c_1(t_*) < \infty$
such that for  index $a$  large enough
\begin{equation} \label{eq cptness requirement}
\varphi_a(S^n \times\{t \}) \subset B^{n+1}_0 (c_1(t_*) ) \quad \text{ for all } t \in [t_*,0].
\end{equation}
We also assume that there is a constant $r_* >0$ such that open ball $B^{n+1}_0(r_*)$ is contained
in the convex hull of $\varphi_a(S^n  \times \{ 0 \})$ for all $a$.
Then there is a subsequence of $\{ \varphi_a \}$ which converges to a strictly convex solution
of Andrews' flow $\varphi_{\infty}: S^n \times (\alpha_
{\infty}, 0] \rightarrow \mathbb{R}^{n+1}$  in any $C^{\infty}$-topology uniformly
on any compact subset  of  $S^n \times (\alpha_{\infty}, 0]$.
\end{theorem}

\begin{proof}
We define map $\bar{\pi}_{a,t}: S^n \rightarrow S^n(1)$ by $z=
\bar{\pi}_{a,t}(p) = \frac{\varphi_a(p,t)}{|\varphi_a(p,t)|}$, clearly the map is one-to-one and onto.
We define the radial distance function
$ r_a( \cdot,t): S^n(1) \rightarrow (0, \infty)$  by $ r_a(z,t) = | \varphi_a(\bar{\pi}_{a,t}
^{-1}(z),t) |$.
Let $\bar{g}$ and $ \bar{\nabla}$ be the Euclidean metric and the Riemannian connection on
$S^n(1)$, respectively.
For a symmetric matrix $B_{n \times n}$ with eigenvalue $b_1, \cdots, b_n$ we define function $F(B)
=f(b_1, \cdots, b_n)$.
By \cite[Lemma 3.2]{An94} we know that $r_a( \cdot, \cdot)$ satisfies the following
 parabolic equation on $S^n(1) \times (\alpha_a, 0]$.
\begin{equation} \label{eq parab r over sph}
\partial_t r_a(z,t) = -F \left ( -\frac{1}{\beta r_a(z,t)^2  } \bar{g}^* \left ( \bar{\nabla} \left (
 \beta \bar{\nabla} r_a(z,t) \right ) \right) + \frac{\operatorname{Id}}{r_a(z,t)} \right ),
\end{equation}
where $\beta =\frac{1}{\sqrt{ | \bar{\nabla}r_a(z,t)|^2 +r_a(z,t)^2}}$.

Fixed a $t_* \in  (\alpha_{\infty}, 0]$, by (\ref{eq cptness requirement}) the solution $r_{a}(\cdot, \cdot)$
is uniformly bounded on $S^n(1) \times [t_*,0]$ for all $a$, i.e.,  the length $r_{a}(z,t) \leq c_1(t_*)$.
Since under the flow the convex hull $\mathcal{K}_{a,t}$ of $\varphi_a (S^n \times\{t \})$
in $\mathbb{R}^{n+1}$ are decreasing in time $t$ (\cite[p.164, line 4]{An94}),
by the convexity of the hypersurfaces and  the assumption
$B^{n+1}_0(r_*) \subset \mathcal{K}_{a,0}$ we conclude that
\[
r_* \leq \langle \varphi_a(p,t), \nu_a(p,t) \rangle = \frac{r_a(z,t)^2}{\sqrt{ | \bar{
\nabla}r_a(z,t)|^2 +r_a(z,t)^2}}.
\]
Hence there is a constant $c_2 = c_2(t_*, r_*)$ such that
 $ |\bar{\nabla}r_a(z,t)| \leq c_2$ on  $S^n(1) \times [t_*,0]$ for all $a$.
This implies that equation (\ref{eq parab r over sph}) is uniformly parabolic on $S^n(1) \times [t_*,0]$.

Because of Assumption A1--A4  the estimates above allow us to apply the Evans-Krylov estimate for
parabolic equations (see, for example, \cite[Theorem 2, p.253]{Kr87}, \cite{An04} for $n=2$) to
(\ref{eq parab r over sph})
and conclude that there is an exponent $\alpha \in (0,1)$ and a constant $c_3 =c_3(t_*, r_*)$ such that
parabolic norm $\|r_a \|_{C^{2+ \alpha, 1+ \alpha/2}(S^n(1) \times [t_*,0]) } \leq c_3$ for all $a$.
Note that the uniform upper bound of high order H\"{o}lder norms $\| r_a \|_{C^{k+ \alpha,
(k+ \alpha)/2}(S^n(1) \times [t_*,0])}$ for each $k > 2$, follows
from the standard  parabolic Schauder theory.
It follows from Arzela-Ascoli theorem that there is a subsequence of $\{ r_a \}$ which converges
to some $r_{\infty}$ in $C^{\infty}$-topology uniformly on $S^n(1) \times [t_*, 0]$.

In the second part of the proof of \cite[Lemma 3.2]{An94} Andrews described how to recover
maps $\varphi_a$ and $\varphi_{\infty}$  (with strictly convex image) from radial length
function $r_a$ and $r_{\infty}$, respectively.
From the discussion it is clear that the subsequence
of $\{ \varphi_a \}$ converges to $\varphi_{\infty}$ in $C^{\infty}$-topology uniformly on $S^n
\times [t_*, 0]$ whenever the corresponding subsequence of $\{ r_a \}$  converges to $r_{\infty}$
smoothly and uniformly.
 Because $t_*$ is chosen arbitrary in $(\alpha_{\infty},0]$, by a diagonalization argument
 the theorem is proved.
\end{proof}

By running the argument about $r_a$ in the proof above only on compact subsets of $S^n(1)$,
it is easy to see the following compactness theorem with possiblely noncompact limits.
Here the base points used for taking the limit are implicitly chosen to be the origin.
 In the proof we  will need the following assumption to
assure the uniform ellipticity of the right hand side of partial differential equation
(\ref{eq parab r over sph}) for the sequence of solutions $r_a$ whose corresponding
hypersurfaces may have principal curvatures arbitrarily close to 0.

\vskip .1cm
\noindent \textbf{Assumption E}. $\frac{\partial f}
{\partial \lambda_i} >0$ on the closure $\overline{\Gamma_+^n} \setminus \{0\}$ for each $i$

\begin{theorem}  \label{thm compactness Andrews flow noncpt limits}
Assume $f$ in  (\ref{eq Andrews flow}) satisfies Assumption E.
Let $\varphi_a: S^n \times (\alpha_a, 0] \rightarrow \mathbb{R}^{n+1}$ be a sequence of
solutions to Andrews' flow (\ref{eq Andrews flow}) with $\lim_{a \rightarrow \infty}
\alpha_a = \alpha_{\infty} <0$. Let functions $r_a( \cdot, \cdot)$ be defined as in the proof
of Theorem \ref{thm compactness Andrews flow} and
let  $\Omega \subset S^n(1)$ be an open subset.
Choose a sequence of smooth compact manifolds with
boundary $\Omega_k \subset \Omega, \, k \in
\mathbb{N}$,  which form an exhaustion of $\Omega$ in the sense that $\Omega_k \subset
\Omega_{k+1}$ and $ \cup_{k} \Omega_k = \Omega$.
 Assume that for any $t_* \in ( \alpha_{\infty},0]$ and $k$ there is a constant $c_1(t_*, k)
< \infty$ such that  for index $a$ large enough
\begin{equation} \label{eq cptness requirement noncpt limit}
r_a(z,t) \leq c_1(t_*, k)  \quad \text{ for all } (z ,t) \in \Omega_k \times [t_*,0].
\end{equation}
We also assume that there is a constant $r_* >0$ such that open ball $B^{n+1}_0(r_*)$ is
contained in the convex hull of $\varphi_a(S^n \times \{ 0 \})$ for all $a$.
Then there is a subsequence of $\{ \varphi_a \}$ which converges to a convex solution
of Andrews' flow $\varphi_{\infty}: \cup_{t \in (\alpha_
{\infty}, 0]} ( \tilde{\Omega}^n_t \times \{t \} )  \rightarrow \mathbb{R}^{n+1}$ in
$C^{\infty}$-topology uniformly on any compact subset
of  $ \cup_{t \in (\alpha_{\infty}, 0]} ( \tilde{\Omega}_t \times \{t \} ) $.
Here domain $\tilde{\Omega}_t = \bar{\pi}_{\infty}( \cdot, t)^{-1}(\Omega) \subset S^n$ where
 map $\bar{\pi}_{\infty}(p,t) = \frac{\varphi_{\infty}(p,t)}{|\varphi_{\infty}(p,t)|} \in  S^n(1) $
 for $p \in \tilde{\Omega}_t$.
\end{theorem}

\section{Construction of ancient solutions}

In this section we prove Theorem \ref{thm main} about the existence of ancient solutions and show
their forward limits are a round point.
At the end we also discuss the non-collapsing property of the solutions.
We leave the properties of the backwards limits to the next section.

\subsection{Construction of the initial hypersurfaces} \label{subsec initial hyper}

Since the existence result in \cite[Theorem 1.1]{An07} requires the initial hypersurfaces to be
 strictly convex,  we need to modify the usual construction of
 initial hypersurfaces (e.g., \cite[p.597]{HH16}) so that their principal curvature are positive everywhere.

Fix an integer $J \in \{ 1, 2, \cdots, n-1 \}$. For each $a \in \mathbb{N}$ we construct a smooth closed
strictly convex hypersurface $\Sigma^n_{J,a,0}
\subset \mathbb{R}^{n+1}$ as follows.
Let $x=(x^1, \cdots, x^{n+1}) =(y,z)$ be coordinates on $\mathbb{R}^{n+1}$ where $y^1 =x^{1},
\cdots, y^{J+1} =x^{J+1}$ and $z^1 =x^{J+2}, \cdots, z^{n-J}=x^{n+1}$.
We define $s =\sqrt{(z^1)^2 + \cdots + (z^{n-J})^2}$.
Fix a constant $\epsilon_0 \in (0,1)$, we choose a rotationally symmetric function $\psi_a$ satisfying
\[
\psi_a(z) =\tilde{\psi}_a(s) =\begin{cases} 1+ \epsilon_0 \quad \text{ if } s=0, \\
1 \quad \text{ if } s=a, \\

0 \quad \text{ if } s=a+1.
\end{cases}
\]
We further require that $\tilde{\psi}_a$ satisfies (i) $\tilde{\psi}_a^{\prime}(s) = \frac{d
\tilde{\psi}_a}{ds} < 0$ for $s \in (0, a+1]$,
(ii) $\tilde{\psi}_a^{\prime \prime}(s) = \frac{d^2 \tilde{\psi}_a}{ds^2}< 0$ for $s \in [0, a+1]$,
(iii) $\tilde{\psi}_a(\tilde{s} +a) =\tilde{\psi}_*(\tilde{s})$ where $\tilde{s} \in [0,1]$,  is a function
independent of $a$, and (iv) embedding
\[
\varphi_{J,a,0}: S^{J} (1) \times B_{0}^{n-J}(a+1) \rightarrow \mathbb{R}^{J+1}
\times \mathbb{R}^{n-J},   \quad \varphi_{J,a,0}(y,z) = (\psi_a(z) y, z)
\]
defines a smooth (including at $|z|=a+1$) closed hyperesurface $\Sigma^n_{J,a,0}$.
Intuitively  $\Sigma_{J,a,0}$ is constructed from capping an almost-cylinder $\{(\psi_a(z) y, z),
\, (y,z) \in S^J(1) \times B^{n-J}_0(a) \}$ off in a $z$-rotationally-symmetric and strictly convex way
(compare with \cite[p.597]{HH16}).
It is easy to see the existence of such functions $\tilde{\psi}_a(s)$.

Now we show that each hypersurface $\Sigma_{J,a,0}$  is strictly convex by computing its second
fundamental form. Because of the $O_{J+1} \times O_{n-J}$-rotational  symmetry,
we use local coordinate $(\tilde{y},z)= (y^1, \cdots, y^{J},
z^1, \cdots, z^{n-J}) \in B^{J}_0 (1) \times B_0^{n-J}(a+1)$ for  $S^{J} (1) \times B_{0}^{n-J}(a+1)$
by taking  $y^{J+1} =\sqrt{1- (y^1)^2- \cdots -( y^{J})^2}$.
Note that $(\tilde{y}, y^{J+1}) \in  S^{J} (1)$.
Let $\{e_{\bar{A}}, \, \bar{A}=1, \cdots, n+1 \}$ be the standard basis of $\mathbb{R}^{n+1}$.
 We have the tangent vectors for $\Sigma_{J,a,0}$ (not necessarily unit)
\begin{align*}
& \tilde{e}_i (\tilde{y},z) =\frac{\partial}{\partial y^i}= \tilde{\psi}_a(s)e_i - \frac{\tilde{\psi}_a
(s)y^i}{y^{J+1}} e_{J+1}, \quad i=1, \cdots, J, \\
& \hat{e}_p (\tilde{y},z) =\frac{\partial}{\partial z^p}= \frac{ \tilde{\psi}_a^{\prime}(s)z^p}{s}
 \sum_{i=1}^{J+1} y^i e_i +  e_{J+1+p}, \quad p=1, \cdots, n-J.
\end{align*}
The unit normal vector of the hypersurface is given by
\[
\nu(\tilde{y},z) = \frac{1}{\sqrt{1+(\tilde{\psi}_a^{\prime}(s))^2}} \left (  \sum_{i=1}^{J+1}
y^i e_i - \frac{\tilde{\psi}_a^{\prime}(s)}{s} \sum_{p=1}^{n-J} z^p e_{J+1+p} \right ) .
\]

Because of the symmetry, below we compute the second fundamental form
at point where $\tilde{y}=0$. We have the directional derivatives in $\mathbb{R}^{n+1}$
\begin{align*}
& D_j \tilde{e}_i (0, z)=\left . \frac{\partial}{\partial y^j} \tilde{e}_i (\tilde{y},z) \right |_{\tilde{y}=0}
=- \tilde{\psi}_a \delta_{ij} e_{J+1} , \quad j=1, \cdots, J,   \\
&D_q \tilde{e}_i (0, z) = \frac{\tilde{\psi}_a^{\prime} z^q}{s} e_i ,  \quad q=1, \cdots, n-J, \\
&D_j \hat{e}_p (0, z) = \frac{\tilde{\psi}_a^{\prime} z^p}{s} e_j  , \\
& D_q \hat{e}_p (0, z) = \left ( \frac{\tilde{\psi}_a^{\prime \prime} z^pz^q}{s^2} + \frac{ \tilde{\psi}
_a^{\prime} (\delta_{pq} s^2 -z^pz^q)}{s^3 } \right ) e_{J+1},
\end{align*}
and the unit normal vector
\[
\nu(0, z) = \frac{1}{\sqrt{1+(\tilde{\psi}_a^{\prime})^2}} \left ( e_{J+1}  - \frac{\tilde{\psi}
_a^{\prime}}{s}  \sum_{p=1}^{n-J} z^p e_{J+1+p} \right ).
\]
Hence the second fundamental form $h(0,z) = \left [ \begin{matrix} - \langle D_j \tilde{e}_i , \nu \rangle
&  - \langle D_q \tilde{e}_i , \nu \rangle   \\
 - \langle D_j \tilde{e}_p, \nu \rangle  &  - \langle D_q \tilde{e}_p, \nu \rangle
\end{matrix} \right ]$ is given
by block matrix
\begin{equation}\label{eq Omega matrix in block}
\left [ \begin{matrix}
\frac{\tilde{\psi}_a}{\sqrt{1+(\tilde{\psi}_a^{\prime})^2}} I_{J} & 0 \\
0 &  -\frac{1}{s^3 \sqrt{1+(\tilde{\psi}_a^{\prime})^2}}
  \Omega_{{(n-J) \times (n-J)}}
\end{matrix} \right ],
\end{equation}
where $I_{J}$ is the identity matrix and matrix $\Omega$ is given by
\[
\left [ \begin{matrix}
s \tilde{\psi}_a^{\prime \prime} (z^1)^2 + \tilde{\psi}_a^{\prime} (s^2 - (z^1)^2)  &  (s
\tilde{\psi}_a^{\prime \prime} - \tilde{\psi}_a^{\prime}) z^1z^2 & \cdots &  (s \tilde{\psi}_a^{
\prime \prime} - \tilde{\psi}_a^{\prime} )  z^1z^{n-J}  \\
 (s \tilde{\psi}_a^{\prime \prime}- \tilde{\psi}_a^{\prime}) z^2z^1 &  s \tilde{\psi}_a^{\prime \prime}
 (z^2)^2 + \tilde{\psi}_a^{\prime} (s^2 - (z^2)^2)  & \cdots &  (s \tilde{\psi}_a^{\prime \prime}
 - \tilde{\psi}_a^{\prime} )  z^2z^{n-J}  \\
  \vdots &  \vdots & \vdots &  \vdots  \\
(s \tilde{\psi}_a^{\prime \prime} - \tilde{\psi}_a^{\prime} )  z^{n-J}z^1   &  (s \tilde{\psi}_a^{
\prime \prime} - \tilde{\psi}_a^{\prime}) z^{n-J}z^2 & \cdots &  s \tilde{\psi}_a^{\prime \prime}
(z^{n-J})^2 + \tilde{\psi}_a^{\prime} (s^2 - (z^{n-J})^2)
\end{matrix} \right ].
\]
We can rewrite the matrix above as
\begin{equation} \label{eq small matrix label}
\tilde{\psi}_a^{\prime} [s^2 I_{n-J} - (z^1, \cdots, z^{n-J})^T
(z^1, \cdots, z^{n-J})]  +s \tilde{\psi}_a^{\prime \prime}  [(z^1, \cdots,
z^{n-J})^T (z^1, \cdots, z^{n-J})].
\end{equation}
It is easy to see that each matrices in the square bracket $[ \,\, ]$ above is negative semi-definite
and to argue that the matrix in (\ref{eq small matrix label}) is negative definite when $z \neq 0$.
It takes a little effort to argue the lower block matrix in (\ref{eq Omega matrix in block}) is positive
definite at $z=0$, actually the block matrix equals to $-\tilde{\psi}_a^{\prime \prime} (0) I_{n-J}$
at $z=0$.

\vskip .1cm
Thus we can conclude the following.

\noindent (C1) Hypersurfaces $\Sigma_{J,a,0}$ are $O_{J+1} \times O_{n-J}$-rotationally symmetric
and strictly convex.

\noindent (C2) The hypersurfaces $\Sigma_{J,a,0}$ are uniformly $(n-J + 1)$-convex, in the sense that
their principal curvatures and mean curvatures satisfy $\lambda_{J,a,0,1} + \cdots +\lambda_{J,a,0,n-J+1}
\geq  \beta H_{J,a,0}$ for some constant $\beta = \beta(n) > 0$ independent of $a$.
To see this at a point where $\tilde{y}=0$, note that length $|\tilde{e}_i(0,z)| = \tilde{\psi}_a(s)$,
hence the principal curvatures induced from the eigenvalues in the first block matrix
in (\ref{eq Omega matrix in block}) are $\frac{1}{\tilde{\psi}_a \sqrt{1+
(\tilde{\psi}_a^{\prime} )^2}}$ of multiplicity $J$.
By the choice of $\tilde{\psi}_a$ we know that $|\tilde{\psi}_a
(\tilde{s}+a) \cdot \tilde{\psi}_a^{\prime}(\tilde{s}+a)| \geq c_1$ for all $a$ and $\tilde{s} \in [0,1]$
 where $c_1>0$ is a constant independent of $a$ and $\tilde{s}$,
 hence these principal curvatures on $\Sigma_{J,a,0}$ satisfy
$c_2 \leq \lambda_{J,a,0,n-J+1} \leq \cdots \leq  \lambda_{J,a,0,n} \leq c_3$ for some
constant $c_2$ and $c_3$ independent of $a$. Hence this implies that there is a uniform
lower bound for $\lambda_{J,a,0,1} + \cdots +\lambda_{J,a,0,n-J+1}$.

Since $|\hat{e}_p(0,z)| = \sqrt{\frac{ (\tilde{\psi}_a^{\prime}(s) z^p)^2}{s^2} +1}$,
the principal curvatures $\lambda_{J,a,0,p}, \, p=1, \cdots, n-J$, are
the eigenvalues of matrix (\ref{eq small matrix label}) multiplied by
$ -\frac{1}{s^3 \sqrt{1+(\tilde{\psi}_a^{\prime})^2}} \cdot \frac{1}{(\tilde{\psi}_a^{\prime})^2
\frac{(z^p)^2}{s^2}+1 }$. Hence these principal curvatures are uniformly bounded from
above by some constant $c_4$ independent of $a$. This implies that there is a uniform upper bound
for the mean curvature $H_{J,a,0}$. Hence the uniform $(n-J + 1)$-convexity is proved.

\noindent (C3) The hypersurfaces $\Sigma_{J,a,0}$ are uniformly non-collapsed from the interior
on the scale of $f$ in the sense that there is a constant $\kappa >0$
independent of $a$ such that for every $x \in
\Sigma_{J,a,0}$ there is an interior sphere tangent to $\Sigma_{J,a,0}$
at $x$ with radius at least $\frac{\kappa}{f(\lambda_{J,a,0,1}(x), \cdots, \lambda_{J,a,0,n}(x))}$.
This is due to the uniform lower bound of function $f(\lambda_{J,a,0,1} (\cdot), \cdots, \lambda_{J,a,0,n}
 (\cdot))$ on  $\Sigma_{J,a,0}$, which is a consequence of Assumption $B_J$.

\subsection{Construction of approximate ancient solutions} \label{subsec approxima ancient}
Let $\Sigma_{J,a,t}^n$ be the Andrews' flow starting from $\Sigma_{J,a,0}$ constructed in \S
\ref{subsec initial hyper} at $t = 0$ (\cite[Theorem 1.1]{An07}).
We know that the flow $\Sigma_{J,a,t}$ collapses to a round point in finite time $T_{J,a}$
and that  $\Sigma_{J,a,t}$ is  $O_{J+1} \times O_{n-J}$-rotationally symmetric by the uniqueness
of the solutions.

We assume that $f$ in the Andrews' flow satisfies Assumption $B_J$.
 Using sphere solution $S^n(R_1(t))$ with $R_1(0) =1$ as an interior barrier and
cylinder solution $S^J(R_2(t)) \times \mathbb{R}^{n-J}$ with $R_2(0) =1+2 \epsilon_0$  as
an exterior barrier,  by  Proposition \ref{prop comparison prin} and Remark \ref{rk comar cylind}
we see that times $T_{J,a}$ are comparable to one in the sense
  $\frac{1}{2 f(1, \cdots, 1)} \leq T_{J,a} \leq \frac{(1+2 \epsilon_0)^2}{2c_{J0}} $.
 Note that the bounds are independent of $a$.

Let $\hat{\Sigma}^n_{J, a, \hat{t}}, \hat{t} \in [\hat{T}_{J,a}, 0)$ (where $\hat{T}_{J,a} < -1$
denotes the new initial time) be the sequence of solutions of (\ref{eq Andrews flow}) obtained by
parabolic rescaling of $\Sigma_{J,a,t}$ with $\hat{t} =\Lambda_{J,a}^2(t- T_{J,a})$ and
 $\hat{\Sigma}_{J, a, \hat{t}} = \Lambda_{J,a} \Sigma_{J, a, t}$ for some $\Lambda_{J,a} >0$.
Define the major radius and the minor radius of hypersurface $\hat{\Sigma}_{J, a, \hat{t}}$  by
\begin{equation} \label{eq amjor minor rad}
A_{J,a}(\hat{t})
= \max_{x \in \hat{\Sigma}_{J, a, \hat{t}}} \left ( \sum_{i=J+2}^{n+1} (x^i)^2 \right )^{1/2},
\quad \text{ and } \quad       B_{J,a}( \hat{t}) = \max_{x \in \hat{\Sigma}_{J,a, \hat{t}}}
\left ( \sum_{i=1}^{J+1} (x^i)^2 \right )^{1/2},
\end{equation}
respectively.
We choose the scaling factor $\Lambda_{J,a}$ such that the ratio $\frac{A_{J,a}(\hat{t})}{
B_{J,a}( \hat{t}) }$  equals $2$ for the first time at $ \hat{t}=-1$.
In the following we will use the ideas from \cite[p.598]{HH16} to prove Claim D1 and D2 stated below.

 \vskip .1cm
Claim D1. There exists a constant $C < 1$ independent of index $a$ such that diameters
\begin{equation}
C \leq \operatorname{diam} ( \hat{\Sigma}_{J,a,-1} ) \leq C^{-1}. \label{eq diam rescaled bdd}
\end{equation}

Proof of Claim D1. Since the flows starting from  $\hat{\Sigma}_{J, a, -1}$
 become extinct in one unit of time, the lower bound
of the diameters follows from comparison with a spherical solution as an exterior barrier of the flows.
For the upper bound, using the $O_{J+1} \times O_{n-J}$-rotational symmetry, $\frac{A_{J,a}(-1)}
{B_{J,a}(-1)} = 2$, and the convexity, we can construct a sphere of radius which is a fraction
(independent of $a$) of the diameters as an interior barrier of $\hat{\Sigma}_{J,a,-1}$.
Since the flow becomes extinct in one unit of time,
by Proposition \ref{prop comparison prin}  this barrier sphere solution extincts within one unit
of time and so the sphere has a radius less or equal  to $\sqrt{2f(1, \cdots, 1)}$.
Hence we have the required upper bound on the diameters.

 \vskip .1cm
Now we show that $\{ \hat{\Sigma}^n_{J, a, \hat{t}} \}_{a=1}^{\infty}$ is a sequence of approximate
ancient solutions of  (\ref{eq Andrews flow}) by proving

Claim D2. $\lim_{a \rightarrow \infty} \hat{T}_{J,a} = -\infty$.

Proof of Claim D2. Fix a time $\hat{t}_0 < -1$. Using the $O_{J+1} \times O_{n-J}$-rotational symmetry,
 $\frac{A_{J,a}(\hat{t}_0)} {B_{J,a}(\hat{t}_0)} \geq 2$, and the convexity,
 we can put a sphere of radius $\frac{B_{J,a}(\hat{t}_0)}{4}$ inside $\hat{\Sigma}_{J,a,
\hat{t}_0}$ at distance $\frac{A_{J,a}(\hat{t}_0)}{2}$ away from the origin.
The sphere is centered on the plane $\{0 \} \times \mathbb{R}^{n-J}$.
Thus by Proposition \ref{prop comparison prin} it takes function $A_{J,a}(\cdot)$ a time period
$|\hat{t}_0 -\hat{t}_*|$ of at least
$\frac{B_{J,a}(\hat{t}_0)^2}{32}$ to decrease from $A_{J,a}(\hat{t}_0)$ to
$\frac{1}{2}A_{J,a}(\hat{t}_0) =A_{J,a}(\hat{t}_*)$.
 On the other hand, $B_{J,a}(\hat{t})$ decreases with time and from Claim D1
we know that $B_{J,a}(-1) \geq \delta$ for some $\delta > 0$ independent of $a$.
Thus, it takes quotient function $\frac{A_{J,a}(\cdot)}{B_{J,a}(\cdot)}$ a time period of
at least $\frac{\delta^2}{32}$ to decrease from $\frac{A_{J,a}(\hat{t}_0)}{B_{J,a}(\hat{t}_0)}$
to $\frac{1}{2} \frac{A_{J,a}(\hat{t}_0)}{B_{J,a}(\hat{t}_0)}$.
Since $\frac{A_{J,a}(\hat{T}_{J,a})}{B_{J,a}(\hat{T}_{J,a})}  \rightarrow \infty$
as $a \rightarrow \infty$ by the construction of
initial hypersurface $\Sigma_{J,a,0}$ and
$\frac{A_{J,a}(-1)}{B_{J,a}(-1)} = 2$,  the claim follows.

\subsection{Limiting the  approximate ancient  solutions} \label{subsec limiting process}
First we verify the assumption after (\ref{thm compactness Andrews flow}) for sequence $\{ \hat{\Sigma}^n_{J,
a, \hat{t}} \}, \hat{t} \in [\hat{T}_{J,a}, -1]$. Recall from the proof at the end of
\S\ref{subsec approxima ancient} we have $B_{J,a}(\hat{t})  \geq B_{J,a}(-1)  \geq \delta$ for all
$\hat{t} \leq -1$ where $\delta > 0$ is a constant independent of $a$.
Since $ \hat{\Sigma}_{J, a, \hat{t}}$ is $O_{J+1} \times O_{n-J}$-rotationally symmetric and convex,
the existence of a fixed size ball inside $\hat{\Sigma}^n_{J, a, \hat{t}}$
 follows from $\frac{A_{J,a}(\hat{t})} {B_{J,a}(\hat{t})} \geq 2$.

To see the bound in (\ref{thm compactness Andrews flow}), we fix a $\hat{t}_* < -1$.
We claim that there is a $c_1(\hat{t}_*)>0$ such that  $B_{J,a}(\hat{t}) \leq c_1(\hat{t}_*)$ for all $J$,
$a$, and $\hat{t}\in [\hat{t}_*,-1]$. To see the claim by contradiction,
we assume that there are sequences $\{a_k \}$ and $\{ \hat{t}_k
\in [\hat{t}_*, -1]\}$  such that  $B_{J,a_k}(\hat{t}_k) \rightarrow \infty$ as $k \rightarrow \infty$,
then the convex hull of $\hat{\Sigma}_{J, a_k, \hat{t}_k}$ contains ball $B_0(\rho_{k})$
where radius $\rho_{k} \rightarrow \infty$.
Applying Proposition \ref{prop comparison prin} to $\hat{\Sigma}_{J,
a, \hat{t}}, \hat{t} \in [\hat{t}_k, -1]$ and the solution with initial surface $S^n(\rho_{k})$,
we conclude that $\hat{t}_k \rightarrow -\infty$ which is a contradiction.

We make another claim that there is a constant $c_2(\hat{t}_*)>0$ such that  $A_{J,a}(\hat{t})
 \leq c_2(\hat{t}_*)$ for all $J$,  $a$, and $\hat{t}\in [\hat{t}_*,-1]$.
To see the claim by contradiction, it follows from $\delta
\leq B_{J,a}(\hat{t}) \leq c_1(\hat{t}_*)$ that we may assume that
there is a sequence  of $\{a_k \}$ and $\{ \hat{t}_k
\in [\hat{t}_*, -1]\}$  such that  $\frac{A_{J,a_k}(\hat{t}_k)}{ B_{J,a_k}(\hat{t}_k)} \rightarrow
\infty$ as $k \rightarrow \infty$, the uniform finite existence time $[\hat{t}_*,-1]$ used for
$\hat{\Sigma}_{J, a, \hat{t}}$ contradicts with the proof of Claim D2, hence the claim is proved.
By the definition of $A_{J,a}(\hat{t})$
and $B_{J,a}(\hat{t})$ we conclude that (\ref{thm compactness Andrews flow}) holds
for sequence  $\{ \hat{\Sigma}_{J,  a, \hat{t}} \}, \, \hat{t} \in [\hat{T}_{J,a}, -1]$.

Now we may use Theorem \ref{thm compactness Andrews flow} to conclude that
sequence $\hat{\Sigma}^n_{J, a, \hat{t}}, \hat{t} \in [\hat{T}_{J,a}, -1]$, subconverges to
some strictly convex  limit  $\hat{\Sigma}^n_{J, \infty, \hat{t}}, \hat{t} \in (-\infty, -1]$
in $C^{\infty}$-topology.
  In the discussion above we may move time $-1$ closer and closer
to $0$, hence by a diagonalization argument we may conclude that we get a limit solution
$\hat{\Sigma}_{J,  \infty, \hat{t}}, \, \hat{t} \in (-\infty, 0)$ of  (\ref{eq Andrews flow}).
$\hat{\Sigma}_{J,  \infty, \hat{t}}$ becomes a round point at origin as $\hat{t} \rightarrow 0$
(\cite[Theorem 1.1]{An07}). The limit is not a round sphere solution because
$\frac{A_{J,a}(-1)}{B_{J,a}(-1)} = 2$ for $\hat{\Sigma}_{J,  \infty, \hat{t}}$.

By the $C^{\infty}$-convergence and the symmetry of  $\hat{\Sigma}_{J, a, \hat{t}}$
  it is clear that $\hat{\Sigma}_{J,  \infty, \hat{t}}$  is $O_{J+1}
\times O_{n-J}$-rotationally symmetric.
This proves Theorem \ref{thm main} about the existence of the ancient soultions of
flow (\ref{eq Andrews flow}).

\begin{remark} $($i$)$ Here we assume that $f$ in Theorem \ref{thm main} is concave.
From the uniformly non-collapsing property of $\Sigma_{J,a,0}$ given in
(C3) near the end of \S \ref{subsec initial hyper} and the scaling invariance of the non-collapsing
property, we may apply \cite[Corollary 3]{ALM13} and conclude that the solutions
$\hat{\Sigma}_{J,a,\hat{t}}$ are uniformly non-collapsed from the interior
on the scale of $f$ for all $a$ and $\hat{t}$. Hence their $C^\infty$-limit
$\hat{\Sigma}_{J,\infty,\hat{t}}$ are uniformly non-collapsed from the interior
on the scale of $f$ for all $\hat{t}$.

  $($ii$)$ Note that for solution $\hat{\Sigma}_{J,  \infty, \hat{t}}$ the ratio of major and
  minor radius  $\frac{A_{J,  \infty}(\hat{t})}{B_{J, \infty}(\hat{t})}  \rightarrow \infty$
  as  $\hat{t} \rightarrow -\infty$.

   $($iii$)$   We define a reflection map $\mathcal{R}: \mathbb{R}^{n+1} \rightarrow \mathbb{R}^{n+1}$ by
\begin{equation} \label{eq reflect}
\mathcal{R}(x^1, \cdots, x^{J+1}, x^{J+2}, \cdots, x^{n+1}) = (x^1, \cdots, x^{J+1}, -x^{J+2},
\cdots, -x^{n+1}).
\end{equation}
It is clear that $\hat{\Sigma}_{J,  \infty, \hat{t}}$ is invariant under the reflection,
$\mathcal{R}(\hat{\Sigma}_{J,  \infty, \hat{t}})=\hat{\Sigma}_{J,  \infty, \hat{t}}$.

  $($iv$)$  Note that if we do not care about the properties in (C2) and (C3), the above
 construction of ancient solutions go through without any change by
 using $\{(y,z ) \in \mathbb{R}^{n+1}, \, |y|^2 + \frac{|z|^2}{a^2} =1 \}$
 as the initial hypersurfaces  $\Sigma_{J,a,0}$.
\end{remark}

\section{The backward asymptotic limits of the ancient solutions}

In this section we consider the backward asymptotic limits of the solutions constructed in
 Theorem \ref{thm main}. In particular, we give a proof of Proposition
 \ref{prop asym back limit cylinder}.  In this section $\hat{\Sigma}^n_{J,  \infty, \hat{t}}, \,  \
 \hat{t} \in ( - \infty, 0)$ denotes the ancient solutions constructed in Theorem \ref{thm main}.

\subsection{Proof of cylinder type  backward asymptotical limits}

We first define the rescaling of  $\hat{\Sigma}^n_{J,  \infty, \hat{t}}, \,  \hat{t} \in ( - \infty, 0)$,
which will be used for taking backward asymptotic limits.
Since $f$ in (\ref{eq Andrews flow}) is homogeneous of degree
one, the parabolically-rescaled $\Sigma^n_{J,K,t}= K^{-1} \hat{\Sigma}_{J,
\infty, K^2 t}$ is still a solution of  (\ref{eq Andrews flow}) for each $K >0$.
We consider the limits of this family of solutions  when $K \rightarrow \infty$.

Let $A_{J,K}(t)$ and $B_{J,K}(t)$ be the major and minor radius  of hypersurface
$\Sigma_{J,K,t}$  as defined in (\ref{eq amjor minor rad}), respectively.
From the proof of Theorem \ref{thm main} we have $\frac{A_{J,K}(t)}{B_{J,K}(t)}
\geq 2$ for $t \leq -1$ and $K \geq 1$, and
$\frac{A_{J,K}(-1)}{B_{J,K}(-1)}  \rightarrow \infty$ as $K \rightarrow \infty$.

Fix a $t_* < -1$, we claim that there is a positive constant $c_1(t_*)$ independent of
 $K \geq 1$ such that $c_1(t_*) < B_{J,K}(t) < c_1(t_*)^{-1}$ for all $t \in [t_*,-1]$ and $K \geq 1$.
First we show that $B_{J,K}(t)$ has a uniform upper bound.
Using the $O_{J+1} \times O_{n-J}$-rotational symmetry, $\frac{A_{J,K}(t)}
{B_{J,K}(t)} \geq 2$, and the convexity, we can construct a sphere of radius
$\frac{1}{n} B_{J,K}(t)$ centered at $0$ as an interior barrier of $\Sigma_{J,K,t}$.
Since the flow starting from $\Sigma_{J,K,t}$ becomes singular within time amount $t_*$,
from Proposition \ref{prop comparison prin} we get that the radius
$\frac{1}{n}B_{J,K}(t)$ is bounded from above by a multiple of $|t_*|^{1/2}$.

To see that $B_{J,K}(t)$ has a uniform lower bound,  using the $O_{J+1} \times O_{n-J}$-rotational
symmetry, $\frac{A_{J,K}(t)} {B_{J,K}(t)} \geq 2$, the convexity, and the reflection invariance,
we can construct cylinder $S^{J}_0(B_{J,K}(t)) \times \mathbb{R}^{n-J}$ as an exterior barrier of
$\Sigma_{J,K, t}$.
Since the flow starting from $\Sigma_{J,K,t}$ becomes singular within time amount $t_*$,
from Remark \ref{rk comar cylind}  we get a lower bound of the radius $B_{J,K}(t)$ by
by a multiple of $|t_*|^{1/2}$.

For the family of solutions $\{ \Sigma_{J,K,t} \}$, we choose $\Omega = S^n(1) \setminus
(\{0 \} \times S^{n-J-1}(1) )$ in Theorem \ref{thm compactness Andrews flow noncpt limits},
it is easy to see that condition (\ref{eq cptness requirement noncpt limit}) holds because of
the uniform upper bound of $B_{J,K}(t)$ proved above.
Assume $f$ in  (\ref{eq Andrews flow}) satisfies Assumption E, we may apply the theorem to
 $\{ \Sigma_{J,K,t} \}$ with $t \in (- \infty, -1]$  and get a subsequential limit  $\Sigma^
 n_{J,\infty, t}, \, t \in (- \infty, -1]$.
This limit is convex and $O_{J+1} \times O_{n-J}$-rotationally symmetric.
It is clear that $\Sigma_{J,\infty, t}$ is invariant under the reflection in (\ref{eq reflect}).

To see the limit is a cylinder, fix a $t \in (-\infty, -1]$,
we choose a point $(0,z_{K,*}) \in \Sigma_{J,K,t}$ and let $\ell_{z_{K,*}}$
be a minimal geodesic in $\Sigma_{J,K,t}$ joining $(0,z_{K,*})$ and
its reflection $\mathcal{R}(0,z_{K,*})$.
From $\frac{A_{J,K}(t)}{B_{J,K}(t)}  \rightarrow \infty$ as $K \rightarrow \infty$,
 we have $A_{J,K}(t) \rightarrow \infty$,
hence the length of geodesic $\ell_{z_{K,*}}$ approaches to
infinity as $K \rightarrow \infty$. Since these geodesics all pass through the ball
$B^{n+1}_0(c_1^{-1}(t))$ due  to the upper bound of  $B_{J,K}(t)$,
these geodesics sub-limit to a line in  $\Sigma_{J,\infty,t}$.
Since $\Sigma_{J,\infty,t}$ has nonnegative sectional curvature,
by combining the Cheeger and Gromoll splitting theorem
and the  $O_{J+1} \times O_{n-J}$-rotational symmetry we conclude that $\Sigma_{J,\infty,t}$
splits as $S^{J}(r_*) \times \mathbb{R}^{n-J}$ for some radius $r_* > 0$.
 Proposition \ref{prop asym back limit cylinder} now follows.

\subsection{A speculation about the bowl type limits}
Choose a sequence of time $\hat{t}_K \rightarrow -\infty$ and a sequence of points
 $(0,z_{K}) \in \hat{\Sigma}_{J,  \infty, \hat{t}_K}$, we define dilation scale $Q_K$ so that
$f (\lambda_{K,1}, \cdots, \lambda_{K,n} ) =1$ where $\lambda_{K,1}, \cdots, \lambda_{K,n}$
are the principal curvatures of the dilated and translated hypersurface $Q_K \left (
\hat{\Sigma}_{J,  \infty, \hat{t}_K} - (0,z_{K}) \right )$ at the origin. Based on the knowledge
about the ancient solutions of mean curvature flow (\cite{An12}, \cite[Theorem 1.1]{HH16}) we would like
to speculate that the family of solutions $ \left \{ Q_K \left (
\hat{\Sigma}^n_{J,  \infty, Q_K^{-2}t+\hat{t}_K} - (0,z_{K}) \right ) \right \}$ would sub-converge
to a solution of the form $\operatorname{Bowl}_t^{J}\times \mathbb{R}^{n-J}$ where
$\operatorname{Bowl}_t^{J}$ is a translating soliton solution of
Andrews' flow (after dimension reduction) (compare \cite{AW94} and \cite{Wh03},  for example).
More precisely, let $\lambda_1(y), \cdots, \lambda_{J}(y)$ be the principal curvature of
$\operatorname{Bowl}_{-1}^{J} \subset \mathbb{R}^{J+1}$ at point $y$, then they satisfy
  \[
f(\lambda_1(y), \cdots, \lambda_{J}(y), 0, \cdots, 0) = - \langle V, \nu_y \rangle
\]
where $V \in \mathbb{R}^{J+1}$ is a fixed vector and $\nu$ is the unit normal direction. For a special choice of $f$ such translating solitons appear in the work of Brendle and Huisken (\cite{BH15}). Also note that the rotational symmetry of such translating solitons are studied by Bourni and Langford (\cite{BL16}).



\begin{thebibliography}{9}

\bibitem[AW94]{AW94} S. Altschuler and L.F. Wu, \emph{Translating surfaces of the non-parametric
mean curvature flow with prescribed contact angle}. Calc. Var., \textbf{2} (1994), 101--111.

\bibitem[An94]{An94} B. Andrews, \emph{Contraction of convex hypersurfaces in Euclidean spaces}.
 Calc. Var.,  \textbf{2} (1994), 151--171.

\bibitem[An04]{An04}  B. Andrews,  \emph{Fully nonlinear parabolic equations in two space variables}.
arXiv:math.DG/ 0402235v1

\bibitem[An07]{An07} B. Andrews, \emph{Pinching estimates and motion of hypersurfaces by
curvature functions}. J. reine angew. Math. \textbf{608} (2007), 17--33.


\bibitem[ALM13]{ALM13} B. Andrews, M. Langford, and J. McCoy, \emph{Non-collapsing
in fully non-linear curvature flow}. Ann. I. H. Poincare - AN \textbf{30} (2013),  23--32.

\bibitem[AMZ13]{AMZ13} B. Andrews, J. McCoy, and Y.  Zheng, \emph{Contracting convex
hypersurfaces by curvature}. Calc. Var. Partial Differential Equations \textbf{47} (2013),  611--665.

\bibitem[An12]{An12} S. Angenent. \emph{Formal asymptotic expansions for symmetric ancient ovals in mean curvature flow}. http://www.math.wisc.edu/angenent/preprints.html, 2012.

\bibitem[ADS18]{ADS18} S. Angenent, P. Daskalopoulos, and N. Sesum. \emph{Uniqueness of two-convex closed ancient solutions to the mean curvature flow}. ArXiv:1804.07230.

\bibitem[BH15]{BH15} S. Brendle and G. Huisken, \emph{A fully nonlinear flow for two-convex hypersurfaces}. arXiv:1507.04651.

\bibitem[BL16]{BL16} T. Bourni and M. Langford, \emph{Type-II singularities of two-convex immersed mean curvature flow}. Geom. Flows, \textbf{2} (2016),  1--17.

\bibitem[BLT17]{BLT17} T. Bourni, M. Langford, and G. Tinglia, \emph{Collapsing ancient solutions of mean curvature flow}. arXiv: 1705.06981.

\bibitem[DHS10]{DHS10} P. Daskalopoulos, R. Hamilton, and N. Sesum, \emph{Classification of
compact ancient solutions to the curve shortening flow}. J. Differential Geom., \textbf{84} (2010),
455--464.

\bibitem[DHS12]{DHS12} {\sc P. Daskalopoulos, R. Hamilton, N. Sesum}, {\em Classification of
compact ancient solutions to the Ricci flow on surfaces}. J. Differential Geom., \textbf{91}
(2012), 171--214.

\bibitem[Ec04]{Ec04} K. Ecker,  \emph{Regularity Theory for Mean Curvature Flow}. Birkhauser,
Basel, 2004.

\bibitem[HH16]{HH16} R.  Haslhofer and O. Hershkovits,  \emph{Ancient solutions of the mean curvature
flow}. Comm. Anal. Geom., \textbf{24} (2016), 593--604.

\bibitem[HK16]{HK16} R. Haslhofer and B. Kleiner, \emph{Mean curvature flow of mean convex
hypersurfaces}. Comm. Pure Appl. Math., \textbf{70} (2017), 511--546.

\bibitem[HIMW]{HIMW} D. Hoffman, T. Ilmanen, F. Martin, and B. White, \emph{Graphical translators for mean curvature flow}. arViv:1805.10860.

\bibitem[Kr87]{Kr87} N.V. Krylov, \emph{Nonlinear elliptic and parabolic equations of the
second order}. D. Reidel, 1987.

\bibitem[Pe02]{Pe02} G. Perelman, \emph{The entropy formula for the Ricci flow
and its geometric applications}. arXiv:math.DG/0211159.

\bibitem[Wa11]{Wa11} X.J. Wang, \emph{Convex solutions to the mean curvature flow}. Ann. of Math. (2),
\textbf{173} (2011), 1185-1239.

\bibitem[Wh03]{Wh03} B. White,  \emph{The nature of singularities in mean curvature flow of mean
convex sets}. J. Amer. Math. Soc., \textbf{16} (2003), 123--138.

\end{thebibliography}
\end{document}